\documentclass{amsart}[12pt]

\newtheorem{theorem}{Theorem}

\newtheorem{proposition}{Proposition}
\usepackage{yfonts}
\usepackage{amssymb}
\usepackage{amsthm}
\usepackage{array}
\usepackage{booktabs}
\usepackage{hhline}
\usepackage{xy}
\usepackage{epsfig}
\usepackage{color}
\usepackage{upgreek}
\usepackage[english]{babel}
\usepackage{epigraph}
\usepackage{fancybox}
\usepackage{shadow}
\usepackage{afterpage}
\usepackage{mathrsfs}
\usepackage{enumitem}

\theoremstyle{definition}

\newtheorem{remark}{Remark}

\theoremstyle{plain}

\newtheorem*{TVM-curvas}{Teorema do Valor Médio (para Curvas)}
\newtheorem*{TVM-aplicacoes}{Teorema do Valor Médio (para Aplicações)}

\newcommand{\sd}{\mathcal{S}_{{\scriptscriptstyle{D}}}}
\newcommand{\bt}{\mathbf{t}}
\newcommand{\bu}{\mathbf{u}}

\newcommand{\R}{\mathbb{R} }

\newcommand{\N}{\mathbb{N} }

\newcommand{\vt}{\vspace{.1cm}}
\newcommand{\vtt}{\vspace{.2cm}}

\renewcommand{\varepsilon}{\epsilon}
\renewcommand{\tau}{\uptau}

\author[J. E. de Freitas, R. F. de Lima, and D. T. dos Santos]{Joaquim E. de Freitas, Ronaldo F. de Lima, \\ and   Daniel T. dos Santos}
\address{Departamento de Matem\'atica -- Centro de Ciências Exatas e da Terra -- Universidade Federal do Rio Grande do Norte.
Lagoa Nova -- Natal RN -- 59030-530.}
\email{jelias@ccet.ufrn.br, ronaldo@ccet.ufrn.br, mathdaniel@hotmail.com}
\title[]{The $n$-dimensional Peano Curve}

\begin{document}

\maketitle

\begin{abstract}
One of the most  startling mathematical discoveries of the nineteen century was the existence of
plane-filling curves. As is well known, the first example of such a curve was given by the Italian mathematician
Giuseppe Peano  in 1890. Subsequently, other examples of plane-filling curves appeared, with  some of them
having  $n$-dimensional analogues. However, the expressions of the coordinates of the
Peano curve are not easily extendable
to arbitrary $n$ dimensions. In fact,
the only known extension of the Peano curve to an $n$-dimensional space-filling curve, made by
Stephen Milne  in 1982, is rather geometric and  makes it difficult
to establish basic properties of these curves, as continuity and nowhere differentiability, as well
as more advanced properties, as  uniform distribution of the coordinate functions.

Here, we will introduce in a completely  analytical
way the $n$-dimensional version of the  Peano curve. More precisely, for a  given integer \,$n\ge 2,$\,
we will define (by means of identities) the \,$n$\, coordinate functions of a continuous and surjective map
from a closed interval to the unit $n$-dimensional cube of \,$\R^n,$\,
which, for the particular case \,$n=2,$\, agrees with the original Peano curve. With this description, as we shall see,
one can easily  establish  all the properties we mentioned above,
and also calculate the Hausdorff dimension of the graphs of  the coordinate functions of this curve.

\end{abstract}

\section{Introduction}

At the end of the nineteenth century, mathematicians were baffled by the appearance of two kinds of
continuous maps: plane-filling curves, and nowhere differentiable functions. The first
example of a plane-filling curve was given by Peano \cite{peano}, in 1890, who introduced  a continuous and surjective map
from an interval onto a square.
Weierstrass, in 1872, provided the first known example of a nowhere differentiable continuous function.

Since Peano's curve became known, many mathematicians --- such as Hilbert, Sierpi\'nski, Lebesgue, and P\'olya ---
obtained examples of plane-filling curves (see, \emph{e.g.}, \cite{sagan}).
Consequently, questions regarding the geometrical and analytical properties of these objects
naturally arose. Many of these questions, in fact, remained unanswered for some time. For instance,
in his paper \cite{peano}, Peano announced  that the coordinate functions of his curve were nowhere
differentiable, but only in 1900 did Moore \cite{moore} give this statement  a complete proof.

The techniques involving the construction of plane-filling curves, in general,
are not easily adaptable to the $n$-dimensional case. For this reason, it also took
some time until  $n$-dimensional space-filling curves appeared.

It was around 1913 when
Hahn and Mazurkiewicz, independently, developed  a method which led to the construction of the $n$-dimensional version
of the aforementioned Lebesgue plane-filling curve\footnote{It should me mentioned that, in contrast with the Peano curve,
the Lebesgue curve is differentiable almost everywhere.}. Other examples were obtained by  Steinhaus \cite{steinhaus},
who proved that $n$-dimensional space-filling curves can be generated by stochastically independent functions.
Nevertheless, none of these curves constitutes a generalisation of the Peano curve.

In \cite{moore}, Moore approached the Peano curve geometrically, rather than analytically, as Peano did.
By adapting Moore's methods, Milne \cite{milne} was able to construct an $n$-dimensional version of the Peano curve, and
to prove that it is measure-preserving.

In this note, we will introduce  an $n$-dimensional space-filling curve which is a fairly simple
extension of the Peano curve. It will be defined by an expression that, for
the  case \,$n=2,$\, coincides with the one that defines the  Peano curve.
For this reason, it will be called the $n$-\emph{dimensional Peano curve}.

We should point out that,  being defined analytically, our $n$-dimensional version of the Peano curve is far more
simple than Milne's. This will allow us to easily establish its fundamental properties of
continuity and  surjectivety,
as well as to characterize each of its coordinate functions
as self-affine (according to K\^ono \cite{kono, kono2}).

By means of this characterization and some results in \cite{kono, kono2},
we will show that, in fact, these coordinate functions are \,$1/n$-H\"older continuous,
uniformly distributed, and  nowhere
$q$-H\"older continuous for \,$q>1/n$\, (in particular, nowhere differentiable).
We will also prove that, as a consequence of being uniformly distributed,
the graphs of these functions have
Hausdorff and packing dimensions both equal to \,$2-1/n.$

\section{Definition of the $n$-dimensional Peano Curve}

For a given integer \,$n>1,$\, let us denote by \,$[0,1]^n$\, the
$n$-dimensional block \,$[0,1]\times\cdots\times [0,1]$\, ($n$\, factors) of the Euclidean space \,$\R^n.$\,
In what follows, we will define a continuous and surjective  map \,$\alpha:[0,1]\rightarrow [0,1]^n.$

Set \,$D=\{0,1,2\}$\, and let \,$\sd$\, denote the set of all sequences in \,$D,$\, that is
$$
\sd =\{(t_k)_{k\in\N} \,;\, t_k\in D \}.
$$

Let  \,$i$\, be an integer such that \,$1\le i\le n.$\, For each \,$\bt=(t_k)_{k\in\N}\in \sd,$\,
consider the subsequence \,$\bt_i=(t_{i+jn})_{j\ge 0}$\,
and define the function \,$x_i:\sd\rightarrow[0,1]$\, by
$$
x_i(\bt)=\sum_{j=0}^{\infty}\frac{\xi^{S_{i+jn}(\bt)}(t_{i+jn})}{3^{j+1}},
$$
where \,$\xi:D\rightarrow D$\, is  the operator
\,$
\xi(a)=2-a,
$\,
and
\begin{equation} \label{eq-def1}
S_{i+jn}(\bt)=\sum_{k=1}^{i+jn}t_k-\sum_{s=0}^{j}t_{i+sn}\,, \,\,\, j\ge 0.
\end{equation}

Notice that, for fixed \,$i$\, and \,$j,$

\begin{itemize}[parsep=1ex]
\item[\Small{\bf (P1)}] \,$S_{i+jn}(\bt)$\, depends only on the first \,$(i-1)+jn$\, terms of \,$\bt,$\, $t_1\,, \dots ,t_{(i-1)+jn}$\,.

\item[\Small{\bf (P2)}]  $\xi^{S_{i+jn}(\bt)}(t_{i+jn})$\, depends only on \,$t_{i+jn}$\, and the parity of
\,$S_{i+jn}(\bt)$\,. More precisely, it  equals  \,$t_{i+jn}$\, if \,$S_{i+jn}(\bt)$\, is even, and
\,$2-t_{i+jn}$\, if \,$S_{i+jn}(\bt)$\, is odd.
\end{itemize}

Property {\Small{\bf (P1)}} is a direct consequence of the definition of \,$S_{i+jn}$\,, and property {\Small{\bf (P2)}} follows from the fact
that the operator \,$\xi$\, is an involution, that is, \,$\xi^2$\, coincides with the
identity map of \,$D.$\,

\begin{remark}
In many of our reasonings concerning the functions \,$x_i$\,, it will be convenient to
represent a given \,$\bt\in\sd$\, in
the following matrix form
$$
[\,\bt\,]=\left[
\begin{array}{ccccc}
t_1     & t_{1+n} & \dots    & t_{1+jn} &  \dots      \\
\vdots & \vdots   &          & \vdots   &   \\
t_i     & t_{i+n} & \dots    & t_{i+jn} &  \dots     \\
\vdots & \vdots   &          & \vdots   &   \\
t_n     & t_{n+n} & \dots    & t_{n+jn} &  \dots
\end{array}
\right].
$$
In this way, \,$S_{i+jn}(\bt)$\,  is the sum of all entries of
\,$[\,\bt\,]$\, from \,$t_1$\, to \,$t_{i+jn-1}$\, (first summand in \eqref{eq-def1}), minus (if \,$j>0$) the sum of the entries
which are located at the $i$-th line, on the left of \,$t_{i+jn}$\, (second summand in \eqref{eq-def1}).
\end{remark}

Define the map
$$
\begin{array}{cccc}
\Phi: & \sd            & \rightarrow & [0,1]\\
      & (t_k)_{k\in\N} & \mapsto     & \sum\frac{t_k}{3^k}
\end{array}
\vt
$$
and, for \,$t\in[0,1],$\,  call each \,$\bt\in\Phi^{-1}(t)$\, a \emph{ternary representation} of \,$t.$\,

Let us prove that, for all \,$i=1,\dots ,n,$\,
\begin{equation} \label{eq-welldefined1}
x_i(\bt)=x_i(\bu) \quad\forall\, \bt, \bu\in\Phi^{-1}(t), \,\, t\in[0,1]\,.
\end{equation}

Indeed, assuming that  \,$\bt$\, and \,$\bu$\, are distinct ternary representations of
\,$t\in [0,1],$\, we can write
\begin{itemize}[parsep=1ex]
\item $\bt=(t_1\,, \dots ,t_{(i_0-1)+j_0n}\,, t_{i_0+j_0n}, \mathbf{0})$\,,

\item $\bu=(t_1\,, \dots ,t_{(i_0-1)+j_0n}\,, t_{i_0+j_0n}-1, \mathbf{2})$\,,
\end{itemize}
where \,$t_{i_0+j_0n}\ne 0$\, and \,$\mathbf{0},$\, $\mathbf{2}$\, denote the constant
sequences equal to \,$0$\, and \,$2,$\, respectively.

Since the first \,$(i_0-1)+j_0n$\, terms of \,$\bt$\, and \,$\bu$\, coincide, it follows from properties
{\Small{\bf (P1)}} and {\Small{\bf (P2)}} that,  for a given \,$i\in\{1,\dots ,n\},$\, the following equality holds
\begin{eqnarray}\label{eq-welldefined}
x_i(\bt)-x_i(\bu) & = & \frac{\xi^{S_{i+j_0n}(\bt)}(t_{i+j_0n})-\xi^{S_{i+j_0n}(\bu)}(u_{i+j_0n})}{3^{j_0+1}}\nonumber \\
                  & + & \sum_{j=j_0+1}^{\infty}\frac{\xi^{S_{i+jn}(\bt)}(0)-\xi^{S_{i+jn}(\bu)}(2)}{3^{j+1}} \cdot
\end{eqnarray}

By considering the matrices of \,$\bt$\, and \,$\bu,$\, one easily concludes that
\,$S_{i+jn}(\bt)$\, and \,$S_{i+jn}(\bu)$\, have distinct parities in any of the following cases:
\begin{itemize}[parsep=1ex]
\item $i<i_0$\, and \,$j>j_0$\,.

\item $i>i_0$\, and \,$j\ge j_0$\,.
\end{itemize}
In particular, for all \,$i>i_0$\,, one has
$$\xi^{S_{i+j_0n}(\bt)}(t_{i+j_0n})-\xi^{S_{i+j_0n}(\bu)}(u_{i+j_0n})=
\xi^{S_{i+j_0n}(\bt)}(0)-\xi^{S_{i+j_0n}(\bu)}(2)=
0.$$
Also, from properties {\Small{\bf (P1)}}  and {\Small{\bf (P2)}},
$$\xi^{S_{i+j_0n}(\bt)}(t_{i+j_0n})-\xi^{S_{i+j_0n}(\bu)}(u_{i+j_0n})=0 \,\,\,\,\, \forall i<i_0\,.$$
These facts, together with equation \eqref{eq-welldefined}, give that
$$
x_i(\bt)=x_i(\bu) \,\, \forall i\ne i_0\,.
$$

Considering again the matrices of \,$\bt$\, and \,$\bu,$\, one sees that, for all
\,$j>j_0$\,:
\begin{itemize}[parsep=1ex]
\item $S_{i_0+jn}(\bt)=S_{i_0+j_0n}(\bt)=S_{i_0+j_0n}(\bu).$

\item $S_{i_0+jn}(\bt)$\, and \,$S_{i_0+jn}(\bu)$\, have the same parity.

\end{itemize}
However,  \,$t_{i_0+j_0n}-u_{i_0+j_0n}=1.$\, Thus,
$$
x_{i_0}(\bt)-x_{i_0}(\bu)=\frac{1}{3^{j_0+1}}-\sum_{j=j_0+1}^{\infty}\frac{2}{3^{j+1}}=0,
$$
if \,$S_{i_0+j_0n}(\bt)$\, is even, and
$$
x_{i_0}(\bt)-x_{i_0}(\bu)=-\frac{1}{3^{j_0+1}}+\sum_{j=j_0+1}^{\infty}\frac{2}{3^{j+1}}=0,
$$
if \,$S_{i_0+j_0n}(\bt)$\, is odd,
which implies
$$x_{i_0}(\bt)=x_{i_0}(\bu)$$
and completes the proof  of \eqref{eq-welldefined1}.

It follows from equality \eqref{eq-welldefined1} that, for \,$i=1,\dots ,n,$\,
the functions
$$
\begin{array}{cccc}
x_i: & [0,1] & \rightarrow & [0,1]\\
     &   t   & \mapsto     & x_i(\bt)
\end{array},
$$
$\bt\in\Phi^{-1}(t),$\, are well defined. Through them,
we will introduce in the next theorem our
intended  $n$-dimensional space-filling curve.

\begin{theorem}
The map
\begin{equation} \label{eq-definition}
\begin{array}{cccc}
\alpha: & [0,1] & \rightarrow & [0,1]^n\\
        &   t   & \mapsto     & (x_1(t),\dots ,x_n(t))
\end{array}
\end{equation}
is continuous and surjective.
\end{theorem}

\begin{proof}
Let \,$t\in[0,1].$\, Given \,$i\in\{1,\dots ,n\}$\, and \,$k\in\N,$\, choose a ternary representation
\,$\bt=(t_s)_{s\in\N}\in\sd$\, of \,$t,$\, in such a way that
$$
t(k):= \sum_{s=1}^{i+kn}\frac{t_s}{3^s}\,\,\, +\sum_{s=(i+1)+kn}^{\infty}\frac{2}{3^s}>t.
$$

It is easily seen that, if \,$t\le u<t(k)$\, and \,$\bu=(u_s)_{s\in\N}$\, is a ternary representation of \,$u,$\, then
the first \,$i+kn$\, terms  of \,$\bt$\, and \,$\bu$\,  coincide.
Therefore,
$$
|x_i(t)-x_i(u)| \le  \sum_{j=k+1}^{\infty}\frac{|\xi^{S_{i+jn}(\bt)}(t_{i+jn})-\xi^{S_{i+jn}(\bu)}(u_{i+jn})|}{3^{j+1}}\le
\sum_{j=k+1}^{\infty}\frac{2}{3^{j+1}}=\frac{1}{3^{k+1}}\,,
$$
which implies that \,$x_i$\, is continuous from the right at \,$t.$\,

An analogous reasoning leads to the conclusion  that \,$x_i$\, is also continuous from the left.
Thus, each coordinate function
of \,$\alpha$\, is continuous, which implies that the map \,$\alpha$\, itself is continuous.

Now, we shall prove that  \,$\alpha$\, is surjective, that is,
for a given point \,$(x_1\,,\dots ,x_n)$\, in \,$[0,1]^n,$\, we will obtain
\,$t\in [0,1]$\, such that
\begin{equation}  \label{eq-surjectivity}
x_i(t)=x_i \,\, \forall i=1,\dots ,n\,.
\end{equation}

Given \,$i\in\{1,\dots ,n\},$\,
let \,$\mathbf{a}_i=(a_{i+jn})_{j\ge 0}\in\sd$\, be a ternary representation
of \,$x_i$\,.
Set  \,$t_1=a_1$\, and, using induction, define for all \,$i=1,\dots ,n$\,
the sequence \,$\bt_i =(t_{i+jn})_{j\ge 0}\in\sd$\,  by the equality
$$
t_{i+jn}=\xi^{S_{i+jn}(\bt_{ij})}(a_{i+jn}),
$$
where
\,$\bt_{ij}=(t_1\,, \dots , t_{(i-1)+jn}\,, \mathbf{0}).$
It follows from property {\Small{\bf (P1)}} that
$$S_{i+jn}(\bt)=S_{i+jn}(\bt_{ij}) \,\, \forall i=1,\dots ,n, \,\, j\ge 0.$$
Therefore, for all \,$i\in\{1,\dots ,n\}$\, and \,$j\ge 0,$\, one has
$$
\xi^{S_{i+jn}(\bt)}(t_{i+jn})=\xi^{S_{i+jn}(\bt)}(\xi^{S_{i+jn}(\bt_{ij})}(a_{i+jn}))=\xi^{2S_{i+jn}(\bt)}(a_{i+jn})
=a_{i+jn}\,,
$$
which implies that \,$t=\Phi(\bt)$\, satisfies \eqref{eq-surjectivity} and, so, that  \,$\alpha$\, is
surjective.
\end{proof}

The map \,$\alpha$\, defined in \eqref{eq-definition} will be called the $n$-\emph{dimensional Peano curve}.
For \,$n=2,$\, it is precisely the plane-filling  curve introduced by Peano in \cite{peano}.
In this case, for \,$i=1,2$\, and \,$j> 0,$\, the functions \,$S_{i+jn}$\, are simply:
\begin{itemize}[parsep=1ex]
\item $S_{1+2j}(\bt)=t_2+\cdots +t_{2j}$\,.

\item $S_{2+2j}(\bt)=t_1+\cdots +t_{2j+1}$\,.
\end{itemize}

So, regarding the construction of the $n$-dimensional Peano curve, our task consisted in finding
suitable functions \,$S_{i+jn}$\,, which would generalize the above
\,$S_{1+2j}$\, and \,$S_{2+2j}$\,.

\section{Properties of the Coordinate Functions of \,$\alpha$}

We now proceed to establish the  properties of the coordinate functions of the $n$-dimensional Peano curve
mentioned at the end of the introduction. They will be derived from  the main results of \cite{kono, kono2},
and  Propositions \ref{prop-selfaffine} and \ref{prop-uniformly} below.

\begin{proposition} \label{prop-selfaffine}
Given \,$k\in\N$\, and \,$1\le i\le n,$\, the $i$-th coordinate function \,$x_i$\, of
the $n$-dimensional Peano curve \,$\alpha$\, satisfies the following relation:
\begin{equation} \label{eq-hausdorff}
{3^k}(x_i(t)-x_i(t(k)))=(-1)^{\sigma(k)}x_i(3^{kn}(t-t(k))),
\end{equation}
where \,$t=\Phi(t_s)_{s\in\N} \in [0,1],$\,
\,$t(k)=\Phi(t_1,\dots, t_{kn}, \mathbf{0})$\,, and \,$\sigma(k)$\, is a nonnegative integer depending on \,$k.$
\end{proposition}

\begin{proof}
Writing  \,$\bt=(t_s)_{s\in\N}$\,, \,$\bt(k)=(t_1,\dots, t_{kn},\mathbf{0}),$\, and \,$\bu=(t_{s+kn})_{s\in\N}$\,,
we observe that
$$
u:=3^{kn}(t-t(k))=\Phi(\bu).
$$

Now, if we set
$$\sigma(k):=\sum_{q=1}^{kn}t_q-\sum_{r=0}^{k-1}t_{i+rn}$$
and consider  the respective matrices of \,$\bt,$\, \,$\bt(k),$\, and \,$\bu,$\, we verify that,
for all \,$i=1,\dots ,n$\, and \,$j\ge 0,$\, the following equalities hold:
$$
\sigma(k)=S_{i+(k+j)n}(\bt(k))=S_{i+(k+j)n}(\bt)-S_{i+jn}(\bu).
$$
Thus, noticing that the first \,$kn$\, terms of \,$\bt$\, and \,$\bt(k)$\, coincide, we have that
\begin{eqnarray}
3^k(x_i(t)-x_i(t(k))) & = &  \sum_{j=k}^{\infty}\frac{\xi^{S_{i+jn}(\bt)}(t_{i+jn})}{3^{j-k+1}}- \sum_{j=k}^{\infty}\frac{\xi^{S_{i+jn}(\bt(k))}(0)}{3^{j-k+1}} \nonumber\\
                      & = & \sum_{j=0}^{\infty}\dfrac{\xi^{S_{i+(j+k)n}(\bt)}(t_{i+(j+k)n})}{3^{j+1}}-\sum_{j=0}^{\infty}\dfrac{\xi^{\sigma(k)}(0)}{3^{j+1}}\nonumber\\
                      & = & \sum_{j=0}^{\infty}\dfrac{\xi^{\sigma(k)}\xi^{S_{i+jn}(\bu)}(t_{i+(k+j)n})}{3^{j+1}}-\sum_{j=0}^{\infty}\dfrac{\xi^{\sigma(k)}(0)}{3^{j+1}}\,\cdot\nonumber
\end{eqnarray}
Therefore,
$$3^k(x_i(t)-x_i(t(k)))=\sum_{j=0}^{\infty}\dfrac{\xi^{S_{i+jn}(\bu)}(t_{i+(k+j)n})}{3^{j+1}}=x_i(u),$$
if $\sigma(k)$ is even, and
$$3^k(x_i(t)-x_i(t(k)))=\sum_{j=0}^{\infty}\dfrac{2-\xi^{S_{i+jn}(\bu)}(t_{i+(k+j)n})}{3^{j+1}}-\sum_{j=0}^{\infty}\dfrac{2}{3^{j+1}}=-x_i(u),$$
if $\sigma(k)$ is odd. In any case, we have
$$3^k(x_i(t)-x_i(t(k)))=(-1)^{\sigma(k)}x_i(3^{kn}(t-t(k))),$$
as we wished to prove.
\end{proof}

Following K\^ono \cite{kono,kono2}, we  say that a function \,$f:[0,1]\rightarrow\R$\, is \emph{self-affine} with scale parameter \,$H>0$\,
to the integer base \,$r\ge 4$\, if,
for any integers \,$k, s$\, satisfying  \,$k\ge 1$\, and \,$0\le s\le r^k-1,$\,
and any  \,$h$\, satisfying \,$0\le h < r^{-k},$\, one has
$$
f(sr^{-k}+h)-f(sr^{-k})={\epsilon(k,s)}{r^{-Hk}}f(r^kh),
$$
where \,$\epsilon(k,s)\in\{-1,1\}.$\,

As pointed out in \cite{kono,kono2}, a self-affine function \,$f:[0,1]\rightarrow\R$\, is not necessarily continuous.
However, if \,$f$\, is continuous with scale parameter \,$H,$\,
then it is $H$-H\"older continuous, that is, there is a constant \,$\lambda>0,$\, such that
$$
|f(t)-f(t')|\le\lambda |t-t'|^H\,\,\, \forall t,t'\in [0,1].
$$

In this setting, if we make \,$r=3^{n},$\, and
\begin{itemize}[parsep=1ex]
\item $t=sr^{-k}+h=\Phi(t_s)_{s\in\N} $\,\,,

\item $t(k)=sr^{-k}=s3^{-kn}=\Phi(t_1\,, \dots ,t_{kn}\,, \mathbf{0})$\,,
\end{itemize}
we get from Proposition \ref{prop-selfaffine} the following result.

\begin{theorem} \label{th-selfaffine}
Any coordinate function \,$x_i$\,  of the $n$-Peano curve \,$\alpha$\, is
self-affine  with  scale parameter  \,$H=1/n$\, to the base \,$r=3^n.$\, In particular,
\,$x_i$\, is $1/n$-H\"older continuous.
\end{theorem}

In the following,  we shall prove that the coordinate functions of the $n$-dimensional Peano curve
are uniformly distributed. With this purpose, we will consider some concepts and
results from \cite{kono2} (see also \cite{urbanski2, urbanski1}), which we will adapt to our context.

We recall that a function \,$f$\, defined in an interval \,$I\subset\R$\, is said to
be \emph{uniformly distributed} (with respect to the Lebesgue measure \,$\mu$) if,  for any measurable
set \,$A\subset\R,$\, \,$f^{-1}(A)$\, is measurable and
\,$\mu(f^{-1}(A))=\mu(A).$

Now, given a nonnegative integer \,$k,$\, set
$$
I_k:=\left[\frac{k}{3^n},\frac{k+1}{3^n}\right).
$$
Define,  for \,$i\in\{1,\dots ,n\}$\, and \,$s\in D=\{0,1,2\},$\,
$$
Q_i(s):=\{k\in [0, 3^n) \,;\, \exists\, t=\Phi(\bt)\in I_k\,, \,\xi^{S_i(\bt)}(t_i)=s\}\,,
$$
and denote the cardinality of \,$Q_i(s)$\, by \,$|Q_i(s)|.$

\begin{proposition} \label{prop-uniformly}
For all   $i\in\{1,\dots, n\},$  the function
\,$s\in D\mapsto |Q_i(s)|$\,
is constant.
\end{proposition}

\begin{proof}
Let us prove first that there is a bijection between \,$Q_i(0)$\, and \,$Q_i(1).$\,
Indeed,
given \,$k\in Q_i(0),$\, let
\,$\bt=(t_{s})_{s\in\N}\in\sd$\,  be such that
$$t=\Phi(\bt)\in I_k \quad\text{and}\quad  \xi^{S_i(\bt)}(t_i)=0.$$
Thus,  \,$t_i=0$\, when \,$S_i(\bt)$\, and  \,$t_i=2$\,
when \,$S_i(\bt)$\, is  odd.

Assume that \,$t_i=0$\, and  define
$$
t'=t+\frac{1}{3^i} \quad\text{and}\quad k'=k+3^{n-i}\,.
$$
Writing \,$t_s'=t_s$\, for  \,$s\ne i,$\, and  \,$t_i'=1,$\,
it is clear that \,$\bt'=(t_s')_{s\in\N}$\,
is a ternary representation of \,$t'.$\, Therefore,
\begin{equation} \label{eq-uniformdistribution}
t'\in I_{k'}=[{k'}/{3^n},(k'+1)/{3^n})\,\,\,  \text{and} \,\,\, \xi^{S_i(\bt')}(t_i')=1.
\end{equation}
Moreover, since \,$t\in I_k$\, and \,$t_i=0,$\, one has
$$
\frac{k}{3^n}\le t<\sum_{q=1}^{i-1}\frac{2}{3^q}+\frac{1}{3^i}=1-\frac{1}{3^{i-1}}+\frac{1}{3^i}\,\cdot
$$
Thus,
$$
\frac{k'}{3^n}=\frac{k+3^{n-i}}{3^{n}}<1-\frac{1}{3^{i-1}}+\frac{2}{3^{i}}=1-\frac{1}{3^{i}}\le 1-\frac{1}{3^n}\,,
\vtt
$$
which implies
\,$k'< 3^n-1.$\,
This, together with \eqref{eq-uniformdistribution}, gives that \,$k'\in Q_i(1).$\,

If \,$t_i=2, $\, we define
$$
t'=t-\frac{1}{3^i}\quad\text{and}\quad k'=k-3^{n-i}\,,
$$
and conclude, analogously, that \,$k'\in Q_i(1).$\,

Now, observe that if  \,$t=\Phi(\bt)\in I_k$\,, the hypotheses
\,$t_i=0$\, and \,$t_i=2$\, are mutually exclusive.
So, in an obvious way,
the family of intervals \,$\{I_k\}_{k\in Q_i(0)}$\, expresses itself as a disjoint
union of two of its subfamilies.
Therefore, the correspondence
$$
\begin{array}{ccc}
Q_i(0) & \rightarrow & Q_i(1)\\
k    & \mapsto     & k'=k\pm \frac{1}{3^i}
\end{array}
$$
is clearly a bijection, where the sign $+$ or $-$ is taken according to the subfamily the interval \,$I_k$\,
belongs to.

In a very similar fashion, we can construct a bijection between \,$Q_i(2)$\, and \,$Q_i(1),$\, which implies
that the function
\,$s\in D\mapsto |Q_i(s)|$\, is, in fact, constant.
\end{proof}

From the definition of the coordinate functions \,$x_i$\,  and the fact that
each of them  is continuous, self-affine,  and
satisfies \,$x_i(0)=0,$\, \,$x_i(1)=1,$\, one concludes that
equation 2.1 in \cite{urbanski2} applies and  yields
$$
\sum_{s=0}^2|Q_i(s)|=3^n\, \,\,\,\, \forall i=1,\dots ,n\,,
$$
which, together with Proposition \ref{prop-uniformly}, gives
$$
|Q_i(s)|=3^{n-1} \,\,\,\, \forall s\in D, \,\,i=1,\dots ,n\,.
$$

From this last equality and Theorem 3 of \cite{kono2}, we
obtain, as intended, the following result.

\begin{theorem} \label{th-uniformlydistributed}
Each coordinate function of the $n$-dimensional Peano curve \,$\alpha$\, is uniformly distributed.
\end{theorem}

It follows from the two preceding theorems  that each coordinate function \,$x_i$\, of
\,$\alpha$\,   fulfills the hypotheses of
Theorems 1 and 2 of \cite{kono}, which leads to our final result.

\begin{theorem} \label{th-fractaldimensions}
For any coordinate function \,$x_i$\, of the $n$-dimensional Peano curve \,$\alpha,$\,  the following
hold:

\begin{itemize}[parsep=1ex]

\item[\emph{i)}] For all \,$q>1/n,$\, $x_i$\, is nowhere \,$q$-H\"older continuous. In particular, \,$x_i$\, is nowhere differentiable.

\item[\emph{ii)}] The Hausdorff and packing dimensions  of the graph of \,$x_i$\, are both equal to \,$2-1/n$\,.
\end{itemize}

\end{theorem}

Regarding property (i), we recall that a  function
\,$f:I\subset\R\rightarrow\R$\, is called $q$-\emph{H\"older continuous at}
\,$t\in I,$\, if  there exist \,$C, \delta>0$\,  such that, for all
\,$t'\in I$\, satisfying \,$|t-t'|<\delta,$\, the inequality
\,$
|f(t)-f(t')|\le C|t-t'|^q
$\,
holds. 

For an account  on  topological dimensions of  certain graphs, including those of coordinate
functions of space-filling curves, we refer
the reader to \cite{allaart-kawamura,kono} and the references therein.


\begin{thebibliography}{99}


\bibitem{allaart-kawamura} Alaart, P. C., Kawamura, K.:
\emph{Dimensions of the coordinate functions
of space-filling curves}. J. Math. Anal. Appl. {\bf 335}, 1161--1176 (2007).

\vt

\bibitem{kono} K\^ono, N.:
\emph{On self-affine functions}. Japan J. Appl. Math. {\bf 3}, 259--269 (1986).

\vt

\bibitem{kono2} K\^ono, N.:
\emph{On self-affine functions II}. Japan J. Appl. Math. {\bf 5}, 441--454 (1988).


\vt

\bibitem{milne} Milne, S. C.:
\emph{Peano curves and smoothness of functions}. Adv. in Math.   {\bf 35}, 129--157 (1980).


\vt

\bibitem{moore} Moore, E.H.:
\emph{On certain crinkly curves}. Trans. Amer. Math. Soc.  {\bf 1}, 72--90 (1900).

\vt

\bibitem{peano} Peano, G.: \emph{Sur une courbe qui remplit toute une aire plane}. Math. Annln. {\bf 36},
157--160 (1890).

\vt

\bibitem{sagan} Sagan, H.:
\emph{Space-filling curves}. Springer-Verlag  (1994).

\vt

\bibitem{steinhaus} Steinhaus, H.: \emph{La courbe de Peano et les fonctions ind\'ependantes}.  C.R. Acad. Sci.,
Paris {\bf 202}, 1961--1963 (1936).




\vt

\bibitem{urbanski2} Urba\'nski, M.: \emph{The probability distribution and Hausdorff
dimension of self-affine functions}.
Probab. Th. Rel. Fields {\bf 84}, 377--391 (1990).

\vt

\bibitem{urbanski1} Urba\'nski, M.: \emph{The Hausdorff dimension of the graphs of continuous self-affine functions}.
Proc. Amer. Math. Soc. {\bf 108}, Number 4, 921--930  (1990).




\end{thebibliography}
\end{document}